\documentclass[reqno,a4paper,11pt]{amsart}
\usepackage[english]{babel}
\usepackage{amssymb}
\usepackage{amsmath}
\usepackage[all]{xy}
\usepackage{mathtools}
\usepackage{braket}
\usepackage{tikz}
\usetikzlibrary{graphs}
\usetikzlibrary{arrows,chains,matrix,positioning,scopes}
\makeatletter
\tikzset{joinlabel/.code=\tikzset{after node path={%
\ifx\tikzchainprevious\pgfutil@empty\else(\tikzchainprevious)%
 edge[every join]#1(\tikzchaincurrent)\fi}}}
\makeatother

\DeclareMathOperator{\spn}{span}
\newtheorem{prop}{Proposition}

\newtheorem{cor}[prop]{Corollary}
\newtheorem{lem}[prop]{Lemma}
\newtheorem{fact}[prop]{Fact}

\newtheorem{ques}{Question}
\newtheorem{ex}[prop]{Example}
\newtheorem*{thmA}{Theorem A}
\newtheorem*{thmA*}{Theorem A*}
\newtheorem*{thmB}{Theorem B}
\newtheorem{rem}[prop]{Remark}

\numberwithin{prop}{section} 
\numberwithin{claim}{prop}
\numberwithin{equation}{section}
\newcommand{\ul}[1]{\underline{#1}}
\newcommand{\biB}{\mathrm{Bi}}
\newcommand{\End}{\mathrm{End}}

\newcommand{\coker}{\mathrm{coker}}
\newcommand{\Hom}{\mathrm{Hom}}

\newcommand{\dExt}{\mathrm{dExt}}

\newcommand{\dH}{\mathrm{dH}}

\newcommand{\ccd}{\mathrm{cd}}

\newcommand{\idn}{\mathrm{ind}}

\newcommand{\Der}{\mathrm{Der}}
\newcommand{\PDer}{\mathrm{PDer}}

\newcommand{\FP}{\textup{FP}}
\newcommand{\ca}[1]{\mathcal{#1}}
\newcommand{\eu}[1]{\mathfrak{#1}}
\newcommand{\CO}{\mathfrak{CO}}
\newcommand{\Z}{\mathbb{Z}}
\newcommand{\Q}{\mathbb{Q}}

\newcommand{\N}{\mathbb{N}}
\newcommand{\E}{\mathrm{E}}
\newcommand{\V}{\mathrm{V}}

\newcommand{\caO}{\ca{O}}

\newcommand{\der}{\partial}
\newcommand{\eps}{\varepsilon}

\newcommand{\RG}{R[G]}

\newcommand{\QG}{\Q[G]}

\newcommand{\dis}{\mathbf{dis}}
\newcommand{\Mod}{\mathbf{mod}}

\newcommand{\QGmod}{{}_{\QG}\Mod}

\newcommand{\QGdis}{{}_{\QG}\dis}

\newcommand{\caR}{\ca{R}}
\newcommand{\eue}{\mathbf{e}}
\newcommand{\caT}{\ca{T}}
\newcommand{\argu}{\hbox to 7truept{\hrulefill}}

\begin{document}

\title[Rational discrete first degree cohomology for t.d.l.c. groups]{Rational discrete first degree cohomology for totally disconnected locally compact groups}
\author{Ilaria Castellano }
\address{University of Southampton, Salisbury Rd, Southampton SO17 1BJ, UK}
  \email[I.~Castellano]{ilaria.castellano88@gmail.com}
  
\maketitle
\begin{abstract}  It is well-known that the existence of more than two ends in the sense of J.R. Stallings for a finitely generated discrete group $G$ can be detected on the cohomology group $\mathrm{H}^1(G,\RG)$, where $R$ is either a finite field, the ring of integers or the field of rational numbers. It will be shown (cf. Theorem~A*) that for a compactly generated totally disconnected locally compact group $G$ the same information about the number of ends of $G$ in the sense of H. Abels can be provided by $\dH^1(G,\biB(G))$, where $\biB(G)$ is the rational discrete standard bimodule of $G$, and $\dH^\bullet(G,\argu)$ denotes rational discrete cohomology as introduced in \cite{it:ratdiscoh}. 

As a consequence one has that the class of fundamental groups of a finite graph of profinite groups coincides with the class of compactly presented totally disconnected locally compact  groups of rational discrete cohomological dimension at most 1 (cf.~Theorem B). 
\end{abstract}

\section{Introduction}\label{s:intro}

\let\thefootnote\relax\footnote{This work was supported by Gruppo Nazionale per le Strutture Algebriche, Geometriche e le loro Applicazioni GNSAGA-INdAM,  by Programma SIR 2014 - MIUR (Project GADYGR) Number RBSI14V2LI cup G22I15000160008 and by EPSRC Grant N007328/1 Soluble Groups and Cohomology.}

For a totally disconnected locally compact (= t.d.l.c.) group $G$ several cohomology theories can be introduced, e.g., the Bredon cohomology with respect to the family of all compact open subgroups of $G$ and the continuous cohomology via cochain complexes. In this paper we investigate the rational discrete first degree cohomology of a t.d.l.c.\! group $G$ as introduced in \cite{it:ratdiscoh}. In Remarks~\ref{rem:bredon} and \ref{rem:continuous} we provide a brief comparison of this cohomology theory with Bredon and continuous cohomology, respectively.

A left $\QG$-module $M$ is said to be \emph{discrete} if the map $\argu\cdot\argu\colon G\times M\to M$ is continuous, where $M$ carries the discrete topology. The category $\QGdis$ of discrete left $\QG$-modules is an abelian category with both enough injectives and projectives. The right derived functors of $\Hom_{\QG}(\argu,\argu)$ have been denoted by $\dExt_{G}^\bullet(\argu,\argu)$, and, for any $k\geq0$, the group  $$\dH^k(G,\argu)=\dExt^k(\Q,\argu)$$ is defined to be the \emph{$k^{th}$ rational discrete cohomolo\-gy group} of $G$ with coefficients in $\QGdis$ (a brief introduction to this cohomology theory and some properties we use further on are given in $\S$\ref{ss:ratdis}).

In this paper we provide several results on the low-dimensional rational discrete cohomology of $G$ by analogy with the discrete case.
In section 3, we show that the functor $\dH^1(G,\argu)$ can be described by means of continuous derivations (cf. Propositions~\ref{prop:1cohomgr}), and consequently by the almost invariant functions when we consider coefficients in a discrete permutation module (cf. Proposition~\ref{prop:almcoh}).

In section 4, we prove the first main theorem of this paper (cf.~Theorem~A*), which provides a cohomological interpretation of Stallings' decomposition theorem for compactly generated t.d.l.c. groups (cf.~Theorem~A). A compactly generated t.d.l.c.\! group $G$ is said to \emph{split non-trivially over a compact open subgroup $K$} if one of the following holds:
    \begin{itemize}
    \item[(S1)] $G$ is a free product with amalgamation $H\ast_K J$, where $H$ and $J$ are compactly generated open subgroups satisfying $K\neq H$ and $K\neq J$;
    \item[(S2)] $G$ is a HNN-extension $H\ast_K^t$ with stable letter $t$, where  $H$ is a compactly generated open subgroup of $G$.
    \end{itemize}
The space of {\it rough ends} of a compactly generated t.d.l.c. group $G$ is defined to be the end space of a rough Cayley graph of $G$ (cf. \cite[\S 3]{kronmoller:rough} and $\S$\ref{ss:rough ends}). Thus the analogue  of Stallings' decomposition theorem for t.d.l.c. groups can be restated as follows.
\begin{thmA}[\protect{\cite[Theorem 13]{kronmoller:rough}}]\label{thmA} Let $G$ be a compactly generated t.d.l.c. group, and let $e(G)$ denote the number of rough ends of $G$. Then the following are equivalent:
\begin{itemize}
\item[(a)] $e(G)>1$, i.e., $G$ has more than one rough end;
\item[(b)] $G$ splits non-trivially over a compact open subgroup. 
\end{itemize}
\end{thmA}

This splitting theorem is essentially due  to Abels \cite[Struktursatz 5.7, Korollar 5.8]{abels} and \cite[\S 3.6]{kronmoller:rough} explains the relation with Abels' work in detail. In particular, it has been shown that the ideal points of the Specker compactification of a compactly generated t.d.l.c. group $G$ can be identified with the rough ends of $G$, which definition here is recalled in $\S 2.2$. 

\medskip

The main purpose of this paper is to give the following cohomological reformulation of Theorem~A.
\begin{thmA*} Let $G$ be a compactly generated t.d.l.c.\! group. Then $(a)$ and $(b)$ of Theorem~A are equivalent to
$$ (c)\quad\dH^1(G,\biB(G))\neq 0.$$
\end{thmA*}
Here $\biB(G)$ denote the \emph{rational discrete standard bimodule} introduced in \cite{it:ratdiscoh} to be a suitable substitute of the group algebra $\QG$ in the context of rational discrete cohomology.  The rational discrete standard bimodule is defined by 
\begin{equation}\label{eq:biB2}
\biB(G)=\varinjlim_{\caO\in\CO(G)} (\Q[G/\caO],\eta_{U,V}),
\end{equation}
where $\CO(G)=\{\,\caO\subset G\mid \caO\ \text{compact open subgroup}\,\}$, and the direct limit is taken along the injective mappings
\begin{equation}
\eta_{U,V}\colon\Q[G/U]\to\Q[G/V],\quad \eta_{U,V}(gU)=\frac{1}{|U:V|}\sum_{r\in\ca{R}}grV,\quad g\in G
\end{equation}
where $V\subset U$ are compact open subgroups of $G$ and $\ca{R}$ is a set of coset representatives of $U/V$. 

Now the new condition (c) guarantees a non-trivial splitting of a compactly gene\-rated t.d.l.c. group by knowing a single cohomology group as \cite[Theorem~IV 6.10]{dicksdun} guarantees for finitely generated discrete groups.  

The presence of the cohomological condition (c) leads us to prove Theorem~A* by means of the chain of implications $(a)\Rightarrow (b)\Rightarrow (c)\Rightarrow (a)$. Clearly, $(a)\Rightarrow (b)$ has been already proven in Theorem~A. Nevertheless, we prefer to provide a similar but substantially different proof (cf. Remark~\ref{kronmoller}).  Moreover, we obtain a new proof for $(b)\Rightarrow (a)$ going through $(c)$  that clarify how the three different aspects of a compactly generated t.d.l.c. group encoded in the conditions (a),(b) and (c) are related.

\medskip 

Stallings' theory of ends for discrete groups had certainly a major impact on geometric group theory. For example, his decomposition theorem - together with Bass-Serre's theory of groups acting on trees - was an essential tool for proving important results on groups of (virtual) cohomological dimension 1 like the Stallings-Swan theorem (cf.~\cite{stal:torsion,swan:coh1}) or the Karrass-Pietrowski-Solitar theorem (cf.~\cite{kps}). In particular, Stallings' decomposition theorem led naturally to the accessibility problem for finitely generated groups. Within the framework of Bass-Serre theory, a finitely generated group is said to be {\it accessible} if it is isomorphic to a fundamental group of a finite graph of groups such that every edge group is finite and every vertex group is a finitely generated group with at most one end. Equivalently, a compactly generated t.d.l.c. group can be defined to be {\it accessible} if it has an action on a tree such that
\begin{itemize}
\item[(A1)] the number of the $G$-orbits on the edges is finite;
\item[(A2)] the edge-stabilizers are  compact open subgroups of $G$;
\item[(A3)] every vertex-stabilizer is a compactly generated open subgroup  of $G$ with at most one rough end.
\end{itemize}
In 1991 M.J. Dunwoody \cite{dun:inaccessible} constructed an inaccessible finitely generated (discrete) group with infinitely many ends. In \cite{kronmoller:rough} the authors related the accessibility of a compactly generated t.d.l.c. group $G$ to the accessibility of some (and hence all)  rough Cayley graph of $G$, which is the analogue of \cite[Theorem 1.1]{random}. In 1985 M.J. Dunwoody \cite{dun:finpres} proved that every finitely presented (discrete) group has to be accessible. The analogue of this result in the context of t.d.l.c. groups is due to Y. Cornulier \cite{corn:qi}. By using this accessibility result, we prove the second main theorem of this paper (cf. Theorem~B).
\begin{thmB}
For every t.d.l.c. group $G$, the following are equivalent:
\begin{itemize}
\item[(i)] the group $G$ is a compactly presented t.d.l.c.\! group with rational discrete cohomological dimension less or equal to one,
\item[(ii)]the group $G$ is isomorphic to the fundamental group $\pi_1(\ca{G},\Lambda)$ of a finite graph of profinite groups $(\ca{G},\Lambda)$.
\end{itemize}
\end{thmB}
This result is the (compactly presented) analogue of \cite[Theorem.~1.1]{dun:acc} that characterizes the (discrete) groups of cohomological dimension at most 1 over a commutative ring $R$ to be fundamendal groups of  graphs of finite groups with no $R$-torsion.
\begin{ques} Is every compactly generated t.d.l.c. group of rational discrete cohomological dimension at most 1 isomorphic to a fundamental group of a graph of profinite groups?
\end{ques}
\subsection*{Acknowledgements:} 
The results presented in this paper are part of the author's PhD thesis. The author would like to thank Th. Weigel for raising the questions discussed in this paper and for some helpful remarks on the manuscript. The author would also like to thank Y. Cornulier for having spotted some mistakes in an earlier version of this paper. 
\section{Preliminaries on ends}\label{s:ends}
\subsection{Graphs.}\label{ss:graphs}
In this paper we use the notion of {\it graph} as introduced by J-P. Serre in~\cite{ser:trees}, i.e., a graph $\Gamma$ consists of a set $\V(\Gamma)$, a set $\E(\Gamma)$ and two maps 
\begin{center}
\begin{tabular}{ll}
$\E(\Gamma)\rightarrow \V(\Gamma)\times \V(\Gamma)$ & $\eue\mapsto (o(\eue),t(\eue)),$\\
$\E(\Gamma)\rightarrow \E(\Gamma)$ & $\eue\mapsto \bar{\eue},$ 
\end{tabular}
\end{center}
satisfying the following condition: for each $\eue\in \E(\Gamma)$ we have $\bar{\bar{\eue}}=\eue,\ \bar{\eue}\neq \eue$ and $o(\eue)=t(\bar{\eue})$. An element $v\in \V(\Gamma)$ is called a \emph{vertex} of $\Gamma$; an element $\eue\in \E(\Gamma)$ is called an (\emph{oriented}) \emph{edge} and $\bar{\eue}$ is its \emph{inverse edge}. The 2-set $\{\,\eue,\bar{\eue}\,\}$ is called a \emph{geometric edge} of $\Gamma$. The vertex $o(\eue)$ is called the \emph{origin} of $\eue$ and the vertex $t(\eue)$ is called the \emph{terminus} of $\eue$. A \emph{path} from a vertex $v$ to a vertex $w$ in $\Gamma$ is defined to be a sequence of edeges $\eu{p}=(\eue_i)_{1\leq i\leq r}$ such that $o(\eue_1)=v, t(\eue_r)=w$ and $t(\eue_i)=o(\eue_{i+1})$ for $i=1,\ldots, r-1$. A path $\eu{p}=(\eue_i)_{1\leq i\leq r}$ is said to be \emph{reduced} if $\eue_i\neq\bar{\eue}_{i+1}$ for every $i=1,\ldots,r-1$.  A reduced path $\eu{p}=(\eue_i)_{1\leq i\leq r}$ satisfying $t(\eue_r)=o(\eue_1)$ is called \emph{circuit of length $r$}, and a \emph{loop} is a circuit of length 1. A graph $\Gamma$ is said to be \emph{connected}, if there exists a path from any vertex $v$ to any other vertex $w$. Every connected subgraph of $\Gamma$ which is maximal with respect to this property is called a \emph{connected component} of $\Gamma$. Thus every graph $\Gamma$ is the disjoint union of its connected components and in this way one defines an equivalence relation $\sim$ on $\ca{V}(\Gamma)$, which is called the \emph{connectedness relation}. A connected non-empty graph without circuits is said to be a \emph{tree}.

\medskip 

For a graph $\Gamma$ we denote by $\ul{\V}(\Gamma)$ the free $\Q$-vector space $\Q[\V(\Gamma)]$ over the set of vertices. If $\Q[\E(\Gamma)]$ denotes the $\Q$-vector space over the set $\E(\Gamma)$ we put
\begin{equation}\ul{\E}(\Gamma)=\Q[\E(\Gamma)]/\spn_\Q\{\,\eue+\bar\eue\mid \eue\in \E(\Gamma)\,\}\end{equation}
the $\Q$-vector space freely generated by the geometric edges of $\Gamma$. Then one has the canonical $\Q$-linear map $\delta:\ul{\E}(\Gamma)\rightarrow\ul{\V}(\Gamma)$ given by
 \begin{equation}\label{eq:delta}
 \delta([\eue])=t(\eue)-o(\eue),\quad \eue\in \E(\Gamma),\end{equation}
 where $[\eue]$ denotes the canonical image of $\eue\in \E(\Gamma)$ in $\ul{\E}(\Gamma)$. Let $\mathrm{H}_\bullet(|\Gamma|;\Q)$ denote the singular homology groups with rational coefficients of the topological realization $|\Gamma|$ of $\Gamma$. One has the following well known result. 
\begin{fact}[{\cite[Corollary 1]{ser:trees}}]\label{fact:ses}
 Let $\Gamma$ be a graph and let $\delta:\ul{\E}(\Gamma)\to\ul{\V}(\Gamma)$ be the map given by \eqref{eq:delta}. Then\\
 
\noindent $(a)\ \ker(\delta)\cong \mathrm{H}_1(|\Gamma|;\Q)$.\\

\noindent $(b)\ \coker(\delta) \cong \Q[\V(\Gamma)/\sim]$, where $\sim$ is the connectedness relation.\\

\noindent In particular,  $\Gamma$ is a tree if, and only if, $\ker(\delta)=0$ and $\coker(\delta)\cong \Q$.
\end{fact}
Thus, given a connected graph $\Gamma$, one has an associated exact sequence
\begin{equation}\label{eq:exseq}
\xymatrix{0\ar[r]&\mathrm{H}_1(|\Gamma|;\Q)\ar[r]&\ul{\E}(\Gamma)\ar[r]^{\delta}&\ul{\V}(\Gamma)\ar[r]&\Q\ar[r]&0}
\end{equation}
of $\Q$-vector spaces. 
\begin{ex}\label{ex:ses2}  Let $G$ be a t.d.l.c. group, and let $\QGdis$ denote  the abelian category whose objects are the discrete left $\QG$-modules (i.e., left $\QG$-modules such that the pointwise stabilizers are open subgroups of $G$).\\
(a) Suppose there exist open subgroups $H,K$ and  $J$ such that $G=H\ast_K J$, i.e., $G$ splits as free product with amalgamation in $K$. The group $G$ is then acting discretely - i.e. with open vertex stabilizers - without edge inversions on a tree with a segment as fundamental domain (cf.~\cite[Theorem 6]{ser:trees}). By applying the orbit-stabilizer theorem, the exact sequence \eqref{eq:exseq} yields
$$\xymatrix{0\ar[r]&\Q[G/K]\ar[r]^-{\delta}&\Q[G/H]\oplus\Q[G/J]\ar[r]&\Q\ar[r]&0},$$
which is a short exact sequence in $\QGdis$.\\
(b) Suppose $G=H\ast_K^t$ is an HNN-extension with stable letter $t$, where $H,K$ are open subgroups of  $G$. Thus $G$ is acting discretely and without edge inversions on a tree with a loop as fundamental domain (cf.~\cite[Remark 1, pg. 34]{ser:trees}). Thus one has the following short exact sequence in $\QGdis$
$$\xymatrix{0\ar[r]&\Q[G/K]\ar[r]^-{\delta}&\Q[G/H]\ar[r]&\Q\ar[r]&0.}$$
\end{ex}
\subsection{The number of rough ends}\label{ss:rough ends}  
A graph $\Gamma$ is said to be \emph{locally finite} if the set $$\mathrm{star}_\Gamma(v)=\{\eue\in \E(\Gamma)|o(\eue)=v\}$$ is finite for every $v\in\V(\Gamma)$. From now on $\Gamma$ will be a connected locally finite graph. For a finite subset $S\subseteq\V(\Gamma)$ let $E_S(\Gamma)=\{\,\eue\in \E(\Gamma)\mid o(\eue)\in S\,\}$, i.e., the union of all $\mathrm{star}_\Gamma(v),\,v\in S$. We denote  by $\Gamma - S$ the subgraph of $\Gamma$ with vertex set $\V(\Gamma)-S$ and  edge set $\E(\Gamma)-(E_S(\Gamma)\cup \overline{E_S(\Gamma)})$, i.e., $\Gamma-S$ is the subgraph obtained from $\Gamma$ by removing $S$ and all the edges attached to $S$. Let $c_S$ be the number of infinite connected components of $\Gamma-S$. For a connected locally finite graph~$\Gamma$
\begin{equation}
e(\Gamma)=\sup\{\,c_S\mid S\subset \V(\Gamma)\ \mbox{finite}\,\}
\end{equation}
will be called the \emph{number of ends} of $\Gamma$. In particular, the graph $\Gamma$ is finite if, and only if, $\Gamma$ is 0-ended.
\begin{fact}\label{fact:cut_end} The number $\E(\Gamma)$ is greater than one if, and only if, there exists an infinite connected subgraph $\ca{C}\subset \Gamma$ such that the set
\begin{equation}\label{eq:cobound}
\delta \ca{C}=\{\,\eue\in \E(\Gamma)\mid \mbox{either}\ o(\eue)\in V(\ca{C})\ \mbox{or}\ t(\eue)\in V(\ca{C})\ \mbox{but not both}\,\}
\end{equation}
is finite and the subgraph $\ca{C}^*=\Gamma- V(\ca{C})$ contains an infinite connected component. 
\end{fact}
The set of vertices $C=V(\ca{C})$ is called a \emph{cut} of $\Gamma$.

\medskip

Recall that two connected graphs $(\Gamma,d_{\Gamma})$ and $(\Gamma',d_{\Gamma'})$ (with the geo\-desic metric) are said to be \emph{quasi-isometric} if there exist a map $\varphi:\V(\Gamma)\rightarrow V(\Gamma')$ and constants $a\geq 1$ and $b>0$ such that for all vertices $v,w\in \V(\Gamma)$
\begin{equation}
a^{-1}d_{\Gamma}(v,w)-a^{-1}b\leq d_{\Gamma'}(\varphi(v),\varphi(w))\leq a\, d_{\Gamma}(v,w),
\end{equation}
and for all vertices $v'\in V(\Gamma')$ one has
\begin{equation}\label{eq:qi2} 
d_{\Gamma}(v',\varphi(\V(\Gamma)))\leq b.
\end{equation}
A map $\varphi$ satisfying the above conditions is called a \emph{quasi-isometry} of graphs. Moreover, the relation of being quasi-isometric is an equivalence relation among graphs and the number of ends is a quasi-isometric invariant (cf. {\cite[Proposition 1]{moller:endsII}}).

\medskip

A t.d.l.c.\! group $G$ is said to be {\it compactly generated } if there exist a compact open subgroup $K$ and a finite symmetric set $S\subset G\setminus K$ such that $G$ is algebraically generated by $S\cup K$. Every such a pair $(K,S)$ will be called a \emph{generating pair} of $G$. The \emph{rough Cayley graph} $\Gamma$ associated to $G$ with respect to the generating pair $(K,S)$ consists of the following data:
\begin{equation} 
\V(\Gamma)=G/K,\quad \E(\Gamma)=\{\,(gK,gsK), (gsK,gK)\mid g\in G,s\in S\,\},
\end{equation}
where the origin and terminus maps are given by projection onto the first and second coordinate, respectively, while the edge inversion mapping permutes the first and second coordinate. 
\begin{rem} In literature these graphs are also known as Cayley-Abels graphs. The definition we have chosen here follows the approach used in \cite[\S2]{kronmoller:rough}, with the difference that the edges of a graph are directed in our setup.
\end{rem}
A rough Cayley graph $\Gamma$ is naturally endowed with a discrete $G$-action, i.e., $G$ is acting with open stabilizers. Moreover, the following fact holds.
\begin{fact}\label{fact:rcg_prop} 
Let $G$ be a compactly generated t.d.l.c. group. Then
\begin{enumerate}
\item[(a)] every rough Cayley graph $\Gamma$ of $G$ is a vertex-transitive, connected and locally finite graph;
\item[(b)]  $G$ has a continuous, proper and cocompact $G$-action on $\Gamma$;
\item[(c)] all rough Cayley graphs of $G$ are quasi-isometric;
\item[(d)] all rough Cayley graphs of $G$ have the same number of ends.
\end{enumerate}
\end{fact}
Thus the \emph{number of rough ends}  $e(G)$ of  a compactly generated t.d.l.c. group $G$ can be defined to be the number of ends of a rough Cayley graph $\Gamma$ associated to $G$ with respect to some generating pair $(K,S)$. 

\begin{ex} (a) If $G$ is a finitely generated discrete group, then the notion of rough Cayley graph gives back the well-known notion of Cayley graph and its number of ends. E.g. $\Z$ and $\mathrm{D}_\infty$ are 2-ended groups. 

(b)
The group $SL_2(\Q_p)$ is a free product with amalgamation of two copies of $SL_2(\Z_p)$. Hence $SL_2(\Q_p)$ has infinitely many rough ends.
\end{ex}
\section{First degree cohomology}
\subsection{Rational discrete cohomology}\label{ss:ratdis} Here we collect some of the properties concerning the rational discrete cohomology for t.d.l.c.\! groups we shall use further on. For the details the reader is referred to \cite{it:ratdiscoh}.

For a t.d.l.c.\! group $G$, let $\QGdis$ denote the abelian full subcategory of $\QGmod$ whose objects are the discrete left $\QG$-modules, i.e., left $\QG$-modules with open stabilizers. The category $\QGdis$ has enough injectives, thus one may define
\begin{equation}
\label{eq:dExt}
\dExt^k_{G}(M,\argu)=\caR^k\Hom_{\QGdis}(M,\argu)
\end{equation}
the right derived functors of $\Hom_{\QG}(M,\argu)$ in $\QGdis$, and the {\it $k^{th}$ discrete cohomology group} of $G$
with coefficients in $\QGdis$ by
\begin{equation}
\label{eq:dExt3}
\dH^k(G,\argu)=\dExt_{G}^k(\Q,\argu),\qquad k\geq 0,
\end{equation}
where $\Q$ denotes the trivial discrete left $\QG$-module.

By using Maschke's theorem, one may prove that the trivial $\QG$-module $\Q$ is projective whenever $G$ is profinite. Consequently, for every t.d.l.c.\! group $G$, the discrete left $\QG$-module $\Q[G/K]$ is projective in $\QGdis$ whenever $K$ is a compact open subgroup of $G$. Moreover, one may stress further this property as follows.

Let $\Omega$ be a left $G$-set with open point-wise stabilizers. Clearly, $\Q[\Omega]$ - the free $\Q$-vector space over the set $\Omega$ - is a discrete left $\QG$-module, which is also called a {\it discrete left $\QG$-permutation module}. 
\begin{prop}[\protect{\cite[Prop. 3.2]{it:ratdiscoh}}]
\label{prop:permod}
Let $G$ be a t.d.l.c. group, and let $\Omega$ be a left $G$-set with compact open stabilizers.
Then $\Q[\Omega]$ is projective in $\QGdis$. 
In particular, the abelian category $\QGdis$ has enough projectives.
\end{prop}
The existence of projective resolutions in $\QGdis$ naturally leads to several finiteness conditions on $G$ as usual. Firstly, the \emph{rational discrete cohomological dimension} of $G$, denoted by $\ccd_{\Q}(G)$, is defined to be the smallest non-negative integer $n$ such that there exists a projective resolution $(P_i,\der_i)$ of $\Q$ in $\QGdis$ of length $\leq n$. Analogously to the discrete case, one has the following properties.
\begin{prop}[\protect{\cite[Prop. 3.7]{it:ratdiscoh}}]\label{prop:cd}
Let $G$ be a t.d.l.c.\! group.
\begin{itemize}
\item[(a)] $G$ is compact if, and only if, $\ccd_{\Q}(G)=0.$
\item[(b)] If $H$ is a closed subgroup of $G$, then
$$\ccd_{\Q}(H)\leq\ccd_{\Q}(G).$$
\end{itemize}
\end{prop}
Moreover, a discrete left $\QG$-module $M$ is said to be {\it finitely generated}, if there exist a finite number of compact open subgroups
$K_1,\ldots,K_n$ of $G$ and an epimorphism
$\pi\colon\coprod_{1\leq j\leq n} \Q[G/K_j]\to M$. Consequently, a discrete left $\QG$-module $M$ is said to be {\it of type $\FP_n$}, $n\geq 0$, if $M$
satisfies one of the following equivalent properties:
\begin{itemize}
\item[(F1)]  there is a partial projective resolution 
$$\xymatrix{
P_n\ar[r]^{\der_n}&P_{n-1}\ar[r]^{\der_{n-1}}&\ldots\ar[r]^{\der_2}&P_1\ar[r]^{\der_1}
&P_0\ar[r]^{\eps}&M\ar[r]&0}$$
of $M$ in $\QGdis$ such that $P_j$ is finitely generated for all $0\leq j\leq n$;
\item[(F2)]  $M$ is finitely generated and for every partial projective resolution 
$$\xymatrix{
Q_k\ar[r]^{\eth_k}&Q_{k-1}\ar[r]^{\eth_{k-1}}&\ldots\ar[r]&Q_1\ar[r]
&Q_0\ar[r]&M\ar[r]&0}$$
in $\QGdis$ with $k<n$ such that $Q_j$ is finitely generated for all $j=0,\ldots,k$, one has that $\ker(\eth_k)$ is finitely generated.
\end{itemize}
E.g.,  $M$ is of type $\FP_0$ if, and only if, $M$ is finitely generated.
If $M$ is of type $\FP_n$ for all $n\geq 0$, then $M$ is called to be {\it of type
$\FP_\infty$}. Accordingly, the group $G$ is said to be of type $\FP_n$, $n\in\N\cup\{\infty\}$,
if the trivial module $\Q$ is of type $\FP_n$ in $\QGdis$.
\begin{prop}\label{prop:fp}
Let $G$ be a t.d.l.c.\! group and $A\in ob(\QGdis)$ of type $\FP_n$, $n\geq 0$. Then for every direct limit $\varinjlim M_\bullet$ in $\QGdis$ the natural homomorphism
$$\varinjlim\dExt_G^k(A,M_\bullet)\to\dExt_G^k(A,\varinjlim M_\bullet),$$ 
is an isomorphism for $k\leq n-1$ and a monomorphism for $k=n$.
\end{prop}
\begin{proof}
 Let $(P_\bullet,\partial_\bullet,\epsilon)$ be a projective resolution of $\Q$ in $\QGdis$ such that $P_j$ is finitely generated for $0\leq j\leq n$. By the $\Hom-\otimes$ identity provided in \cite[\S 4.3]{it:ratdiscoh}, $\Hom_G(P_j,\argu)$ commutes with direct limits whenever $0\leq j\leq n$. Thus the proof of \cite[Prop. 1.2]{bieri} can be transferred here.
\end{proof}
\begin{rem}\label{rem:fp} If all of the canonical maps $M_\bullet\to\varinjlim M_\bullet$ are injective, then an easy diagram chasing shows that $\varinjlim\dExt_G^n(A,M_\bullet)\to\dExt^n(A,\varinjlim M_\bullet)$ is an isomorpshism as well.
\end{rem}
\begin{cor}\label{cor:fpinf} For a t.d.l.c.\! group $G$ of type $\FP_\infty$ the functors $\dH^\bullet(G,\argu)$ commute with direct limits in $\QGdis$.
\end{cor}
It is well known that a discrete group is fini\-te\-ly generated if, and only if, it is of type $\FP_1$ (cf. \cite[\S VIII.4]{brown:coh}). The analogue result for t.d.l.c.\! groups holds as well.
\begin{prop}[\protect{\cite[Prop. 5.3]{it:ratdiscoh}}]\label{prop:fp1}
Let $G$ be a t.d.l.c.\! group. Then $G$ is compactly generated if, and only if, $G$ is of type $FP_1$.
\end{prop}
By combining the latter finiteness conditions, one defines a t.d.l.c. group $G$ to be {\it of type $\FP$},
if  $G$ is of type $\FP_\infty$ with $\ccd_\Q(G)=d<\infty$. In other words,
the trivial left $\QG$-module $\Q$ has a projective resolution which is finitely generated and concentrated in degrees $0$ to $d$. 

\medskip 

Since the group algebra $\QG$ is not a discrete $\QG$-module unless the group $G$ itself is discrete, in \cite{it:ratdiscoh} a possible substitute has been introduced and studied. Namely,  the rational discrete standard bimodule $\biB(G)$ (cf. \eqref{eq:biB2}). The following are in analogy with the discrete case.
\begin{fact}[\protect{\cite[Prop. 4.3]{it:ratdiscoh}}]\label{fact:bg}
Let $G$ be a t.d.l.c.\! group. 
One has
$$\Hom_G(\Q,\biB(G))\simeq
\begin{cases}
\Q&\ \text{if $G$ is compact,}\\
0&\ \text{if $G$ is not compact.}
\end{cases}$$
\end{fact}
\begin{prop}[\protect{\cite[Prop. 4.7]{it:ratdiscoh}}]
\label{prop:FPbg}
Let $G$ be a t.d.l.c. group of type $\FP$. Then
\begin{equation}
\label{eq:FP-1}
\ccd_\Q(G)=\max\{\,k\geq 0\mid \dH^k(G,\biB(G))\not=0\,\}.
\end{equation}
\end{prop}
\subsection{Derivations}\label{ss:der} Let $\Der(G,M)$ denote the group of all (algebraic) derivations $d$ from a group $G$ to a left $G$-module $M$, i.e., $d$ is a mapping of sets $d\colon G\to M$ satisfying $d(gh)=gd(h)+d(g)$ for all $g,h\in G$. 

For a t.d.l.c. group $G$ and a discrete $\QG$-module $M$, we define
\begin{eqnarray}\label{eq:der}
\Der_K(G,M)&=&\set{d\in \Der(G,M)\mid d(k)=0,\ \forall k\in K},\\
\PDer_K(G,M)&=&\set{d\in \Der_K(G,M)\mid \exists m\in M^K\ s.t.\ d(g)=gm-m\ \forall g\in G},\nonumber
\end{eqnarray}
where $K$ is a compact open subgroup of $G$. Clearly every element $d$ of $\Der_K(G,M)$ is a continuous map, where $M$ carries the discrete topology.

By analogy to the discrete case, one may prove the following result and we include the standard proof for reader's convenience.
\begin{prop}\label{prop:1cohomgr} 
For a compact open subgroup $K$ of a t.d.l.c. group $G$ there is a natural isomorphism 
$$\dH^1(G,M)\cong \Der_K(G,M)/\PDer_K(G,M),$$
where $M\in ob(\QGdis)$.
\end{prop}
\begin{proof} Let 
\begin{equation}\label{eq:aug ses}
\xymatrix{0\ar[r] &N\ar[r]&\Q[G/K]\ar[r]^{\varepsilon}&\Q\ar[r]&0}
\end{equation}
be the short exact sequence in $\QGdis$ provided by the augmentation map $\varepsilon$. Thus the set $\{gK-K\mid g\in G\setminus K\}$ is a generating set of $N$ as $\Q$-vector space. Firstly, notice that $$\Hom_G(N,M)\cong\Der_K(G,M),$$for every $M\in ob(\QGdis)$. Indeed for every $\QG$-map $\varphi\colon N\to M$ let 
\begin{equation}\label{eq:formula}
D_{\varphi}:G\rightarrow M,\quad D_{\varphi}(g)=\varphi(gK- K)\quad \forall g\in G.
\end{equation}
Clearly, $D_{\varphi}\in\Der_{ K}(G,M)$. Thus the formula \eqref{eq:formula} defines a natural homomorphism from  $\Hom_G(N,M)$ to $\Der_{ K}(N, M)$. This homomorphism admits the  inverse  $D\mapsto \varphi_D$ given by $\varphi_D(gK-K)=D(g)$, which is well-defined since $D\in\Der_K(G,M)$ is constant on the cosets of $K$ in $G$.

By applying the long exact cohomology functor to \eqref{eq:aug ses} with coefficients in $M$, one has
\begin{equation}\label{eq:aug les}
\xymatrix{0\ar[r]&M^G\ar[r]&M^K\ar[r]&\Der_K(G,M)\ar[r]&\dH^1(G,M)\ar[r]&0},
\end{equation}
since $\Q[G/K]$ is projective in $\QGdis$
 and  $\Hom_G(\Q[G/K],M)\cong\Hom_K(\Q,M)$ (cf. Proposition~\ref{prop:permod} and \cite[\S 2.9]{it:ratdiscoh}). Finally, as $\PDer_K(G,M)\cong M^K/M^G$ by definition, \eqref{eq:aug les} yields the claim.
\end{proof}
\begin{cor}\label{cor:contcoho} For a t.d.l.c.\! group $G$ and  $M\in ob(\QGdis)$, let $\Der_{top}(G,M)$ be the group of all continuous derivations from $G$ to $M$ and $\PDer_{top}(G,M)$ the subgroup of the principal one. Thus
$$\dH^1(G,M)\cong \Der_{top}(G,M)/\PDer_{top}(G,M),$$
naturally.
\end{cor}
\begin{proof}
Let $d$ be a continuous derivation from $G$ to $M$. Then $\rho:G\times M\to M$ given by $\rho(g,m)=gm+d(g)$ defines a continuous affine transformation of $M$. For every compact open subgroup $K$ of $G$, the $K$-orbit is finite, by continuity. So the average of this orbit is a $K$-fixed point, say $x$. Let $d'\in \PDer_K(G,M)$ be the principal derivation associated to $x$. Since $d-d'\in\Der_K(G,M)$, every continuous 1-cocycle is cohomologous to a 1-cocycle vanishing on $K$. 
\end{proof}

\begin{rem}\label{rem:continuous}
One can chose to develop a cohomology theory for a t.d.l.c. group $G$ directly via cochain complexes. For $M\in ob(\QGdis)$ let $C^n(G,M)$ be the set of all continuous functions from $G^n$ to $M$, where $M$ carries the discrete topology. By equipping this with the usual coboundary operators, one has a cochain complex whose cohomology can be defined to be the {\it continuous cohomology } of $G$, e.g. \cite{hochmost,moore}. By Corollary~\ref{cor:contcoho}, the rational discrete cohomology of $G$ turns out to be equivalent to the continuous one  in degree 0 and 1, but at this stage we do not know if this is true for $n\geq2$.
\end{rem}
\begin{rem}\label{rem:bredon} Let $\ca{C}$ be the family of all compact open subgroups of a t.d.l.c. group $G$. By van Dantzig's Theorem, $\ca{C}$ is non-empty. Furthermore $\ca{C}$ is closed under conjugation and taking finite intersections. Let $\ca{O}_{\ca{C}}(G)$ be the orbit category of $G$ w.r.t. $\ca{C}$. Namely, the objects are the $G$-sets $G/K$, for $K\in\ca{C}$, and the morphisms are the $G$-maps between them. Thus one may define the category of Bredon modules over $\ca{O}_{\ca{C}}(G)$ as usual. The Bredon cohomology of $G$ is not equivalent to the rational discrete cohomology of $G$. Indeed a necessary condition for a t.d.l.c. group $G$ to be of type $\FP_0$ in the Bredon cohomology is the following: there are finitely many compact open subgroups $K_1,\ldots,K_n$ of $G$ such that any compact open subgroup of $G$ is subconjugated to one of the $K_i$s (cf. \cite[Lemma 2.3]{nun}). On the other hand, being of type $\FP_0$ for a t.d.l.c. group in the rational discrete cohomology is an empty condition.
\end{rem}

\begin{rem}
We are aware of a possible connection between rational discrete cohomology and the cohomology of the Hecke algebra (cf. \cite[\S 2]{hecke}) but it will be not discussed in this paper.
\end{rem}

\subsection{The almost invariant functions}\label{ss:aif h1}  In order to connect the rational discrete cohomology of $G$ to the number of rough ends as clearly as possible,  we provide another representation of $\dH^1(G,M)$  whenever $M$ is a transitive discrete permutation module. 

Let $G$ be a compactly generated t.d.l.c. group and let $(K,S)$ be a generating pair of $G$. Clearly, the set $\Hom_\Q(\Q[G/K],\Q)$ of all functions from $G/K$ to $\Q$ is a $G$-set with action given by
\begin{equation}
(g\cdot\alpha)(x)=\alpha(g^{-1}x)\ \ \forall\alpha\in\Hom_\Q(\Q[G/K],\Q),\ \forall g\in G,\ \forall x\in G/K.
\end{equation}
Following \cite{dun:acc}, we say that two maps $\alpha,\beta\in \Hom_\Q(\Q[G/K],\Q)$ are  \emph{almost equal}, and denote this by $\alpha=_a\beta$, if $\alpha(x)=\beta(x)$ for all but finitely many elements $x\in G/K$. 
\begin{ex}\label{ex:quo_ainv}
Every element $m\in\Q[G/K]$ can be expressed as formal sum $$m=\sum_{x\in G/K} q_x x$$ with $q_x\in\Q$ being 0 for almost all $x\in G/K$. Then $m$ can be identified with the projection $p_m:G/K\rightarrow\Q$ given by $p_m(x)=q_x$, showing that $p_m=_a 0$. Thus $\Q[G/K]$ is the set of all \emph{almost zero functions} in $\Hom_\Q(\Q[G/K],\Q)$.
\end{ex}
An element $\alpha\in \Hom_\Q(\Q[G/K],\Q)$ is called an \emph{almost $(G,K)$-invariant function} if $g\cdot\alpha=_a \alpha$ for all $g\in G$ and $k\cdot\alpha=\alpha$ for all $k\in K$.
Denote by $\ca{A}Inv_K(G,\Q)$ the space of all almost $(G,K)$-invariant functions. 
\begin{prop}\label{prop:almcoh} 
For every compact open subgroup $K$ of a t.d.l.c. group $G$ one has
$$\dH^1(G,\Q[G/K])\cong \frac{\ca{A}Inv_K(G,\Q)}{C(G/K)+\Q[G/K]^K},$$
where $$C(G/K)=\{\alpha\in \Hom(\Q[G/K],\Q)|\alpha \mbox{\ constant}\},$$ and $\Q[G/K]^K$ denotes the largest $K$-invariant submodule of $\Q[G/K]$.
\end{prop}
\begin{proof}
The second part of the proof of Lemma 1.1 in \cite{dunbam} can be easily adapted to our context. Thus for every compact open subgroup $K$ of $G$ there exists the following short exact sequence 
\begin{equation}
0\longrightarrow C(G/K)\longrightarrow \ca{A}Inv_K(G,\Q)\stackrel{\partial}{\longrightarrow} \Der_K(G,\Q[G/K])\longrightarrow 0,
\end{equation}
where for each $\alpha$ the map $\partial\alpha\colon G\rightarrow \Q[G/K]$ is given by $\partial\alpha(g)=g\cdot\alpha-\alpha$.

As $\PDer_K(G,\Q[G/K])\cong\Q[G/K]^K$, applying Proposition~\ref{prop:1cohomgr} concludes the proof.\end{proof}

\section{The decomposition theorem}

The aim of this section is to prove Theorem~A*.  Clearly, the proof of Theorem~A* can be shortened considering that the equivalence between a) and b) is well-known, but here we prove the result via the chain of implications $(a)\Rightarrow (b)\Rightarrow (c)\Rightarrow (a)$.

 Recall that a t.d.l.c. group $G$ acts {\it discretely} on a graph if the stabilizers are open subgroups of $G$.
\begin{prop}\label{prop:compgen}
Let $G$ be a compactly generated t.d.l.c. group. Suppose that $G$ acts discretely on a tree $\ca{T}$ such that
\begin{itemize}
\item[(i)] the group $G$ is acting without edge inversions;
\item[(ii)] the quotient graph  $G\backslash\ca{T}$ is finite;
\item[(iii)] the edge stabilizers $G_\eue$ are compact open subgroups of $G$.
\end{itemize}
Then the vertex stabilizers $G_v$ are compactly generated.
\end{prop}
\begin{proof}
 Recall that a t.d.l.c. group is compactly generated if, and only if, it is of type $\FP_1$ (cf. Prop.~\ref{prop:fp1}). Thus it is sufficient to prove that the trivial module $\Q$ is of type $\FP_1$ in $_{\Q[G_v]}\dis$, for all $v\in V(\ca{T})$. By property (i) and \eqref{eq:exseq}, one has that the following sequence
\begin{equation}
\label{eq:ses1}
\xymatrix{
0\ar[r]&\coprod_{\eue\in\ca{R}_{E}(\caT)} \Q[G/G_\eue]\ar[r]&
\coprod_{v\in\ca{R}_{V}(\caT)} \Q[G/G_v]\ar[r]&\Q\ar[r]&0
}
\end{equation}
is exact in $\QGmod$, where $\caR_{V}(\caT)$ is a set of representatives of the $G$-orbits on $V(\ca{T})$,
and  $\caR_{E}(\caT)$ is a set of representatives of the $C_2\times G$-orbits on $E(\caT)$. In particular, $\caR_{V}(\caT)$ and $\caR_{E}(\caT)$ are finite by (ii). Moreover, $G$ is acting discretely on $\ca{T}$, i.e. with open stabilizers, thus \eqref{eq:ses1} is a short exact sequence in $\QGdis$. Thus one may consider the induction functors $\idn_{G_*}^G\colon_{\Q[G_*]}\dis\to_{\Q[G]}\dis$, where $*\in\{v,\eue\mid v\in\caR_V(\caT),\eue\in\caR_E(\caT)\}$ (cf. \cite[\S 2.4]{it:ratdiscoh}). In particular \eqref{eq:ses1} can be reformulated as follows
\begin{equation}
\label{eq:ses2}
\xymatrix{
0\ar[r]&\coprod_{\eue\in\ca{R}_{E}(\caT)} \idn_{G_\eue}^G(\Q)\ar[r]&
\coprod_{v\in\ca{R}_{V}(\caT)} \idn_{G_v}^G(\Q)\ar[r]&\Q\ar[r]&0.
}
\end{equation}
For $G$ is a compactly generated t.d.l.c. group, the trivial module $\Q$ is of type $\FP_1$ in $\QGdis$. The permutation module $\coprod_{\eue\in\ca{R}_{E}(\caT)} \idn_{G_\eue}^G(\Q)$ with compact open stabilizers is a finitely generated projective discrete $\QG$-module, and so of type $\FP_1$ as well (cf. Proposition~\ref{prop:permod}). By applying the horseshoe lemma to \eqref{eq:ses2}, one has that $\coprod_{v\in\ca{R}_{V}(\caT)} \idn_{G_v}^G(\Q)$ is of type $\FP_1$ in $\QGdis$. Hence $\idn_{G_v}^G(\Q)$ is of type $\FP_1$ for every $v\in\ca{R}_{V}(\caT)$ (cf. \cite[Prop. 1.4 (a)]{bieri}). As the induction functor is exact and it is mapping projectives to projectives (cf. \cite[Proposition 3.4]{it:ratdiscoh}), one deduces that the trivial module $\Q$ is of type $\FP_1$ in $_{\Q[G_v]}\dis$, since $\idn_{G_v}^G(\Q)$ is of type $\FP_1$ in $\QGdis$ for every $v\in\ca{R}_{V}(\caT)$.
By conjugation, (iv) holds.
\end{proof}
\begin{rem} It is possible to extend the previous result to actions with edge inversions. In such a case, one has to consider the stabilizers $G_{\{\eue\}}$ of the geometric edges $\{\eue,\bar{\eue}\}$ (cf. \cite[Prop. 5.4]{it:ratdiscoh}). On the other hand, it is well-known that the condition about the action without edge-inversions is not properly a restriction, since it is always possible to consider the barycentric subdivision of the tree. 
\end{rem}

\begin{proof}[Proof of $(a)\Rightarrow (b)$]
Starting from a rough Cayley graph associated to $G$, one may use different techniques to construct a tree satisfying the hypothesis in the previous result  whenever $G$ has more than one rough end (cf. \cite{dunkron:vertexcuts}, \cite{dicksdun}). Thus the result follows by Proposition~\ref{prop:compgen} and Bass-Serre theory.
\end{proof}
\begin{rem}\label{kronmoller}
In \cite{kronmoller:rough}, to prove that a compactly generated t.d.l.c. group $G$ with more than one rough end splits non-trivially over a compact open subgroup (namely, $(a)\Rightarrow (b)$) the authors applied the following technique. Firstly, by using the theory of structure trees developed in \cite{dicksdun}, they construct a directed tree acted on by $G$ with finitely many orbits such that the edge stabilizers are compact and open and the vertex stabilizers are (open) subgroups of $G$. Secondly, they applied Bass-Serre theory of groups acting on trees to conclude that $G$ has to split. Finally, they had to prove that every vertex stabilizer $G_\alpha$ is compactly generated, which is the main part of the proof.  They achieve this final step by constructing a connected locally finite graph acted on transitively by $G_\alpha$ with compact open stabilizers (cf. \cite[Theorem 1]{kronmoller:rough}). This graph is obtained by means of a construction developed in \cite[Section 7]{random}. By Proposition~\ref{prop:compgen} instead, one directly deduces that the vertex stabilizers are compactly generated.
\end{rem}

\begin{proof}[Proof of $(b)\Rightarrow (c)$]
 Let $G$ split non-trivially over the compact open subgroup $K$, i.e., either (S1) or (S2) holds. The proof is split up as follows.

\noindent \textbf{Case 1.} According as the splitting type (i.e., either (S1) or (S2)), suppose $H$ and $J$ are both compact. By Bass-Serre's theory, $G$ is acting on the universal covering tree $\tilde\Gamma$, thus \eqref{eq:exseq} yields a short exact sequence 
\begin{equation}\label{eq:2--3}
\xymatrix{0\ar[r]&\ul{\E}(\tilde\Gamma)\ar[r]^{\delta}&\ul{\V}(\tilde\Gamma)\ar[r]&\Q\ar[r]&0},
\end{equation}
(cf. Example~\ref{ex:ses2}).
Since the vertex stabilizers are conjugated to $H$ (and $J$ respectively), $G$ is acting on $\tilde\Gamma$ with compact open stabilizers. Hence \eqref{eq:2--3} is a projective resolution of $\Q$ of length 1 in $\QGdis$, since it has discrete permutation $\QG$-modules with compact stabilizers in degree 0 and 1 (cf. Proposition~\ref{prop:permod}). Therefore $\ccd_{\Q}(G)=1$, as $G$ is non-compact (cf. Proposition~\ref{prop:cd}(a)). By Proposition \ref{prop:fp1}, since $G$ is compactly generated, $G$ is a t.d.l.c. group of type $\FP_1$ with $\ccd_ {\Q}(G)=1$,  so $G$ is of type $\FP$. Thus Proposition~\ref{prop:FPbg} yields the claim.

\noindent\textbf{Case 2.} Assume $G=H*_K^t$ and $H$ is non-compact. As shown in Example~\ref{ex:ses2}(b), one has the following short exact sequence in $\QGdis$
\begin{equation}
\xymatrix{0\ar[r]&\Q[G/K]\ar[r]^\delta&\Q[G/H]\ar[r]&\Q\ar[r]&0.}
\end{equation}
Recall that for every compact open subgroup $\caO$ of $G$ one has $$\Q[G/\caO]\cong\QG\otimes_{\Q[\caO]}\Q=\idn_\caO^G(\Q),$$ where $\idn_{\ca{O}}^G(\argu):{}_{\Q[\caO]}\dis\to\QGdis$ is the induction functor (cf. \cite[\S2.4]{it:ratdiscoh}). By the Eckmann-Shapiro type lemma \cite[\S2.9]{it:ratdiscoh}, applying the long exact cohomology functor with coefficients in $\biB(G)$ yields the long exact sequence
\begin{equation}\label{eq:les}
\xymatrix{0\ar[r]&\biB(G)^G\ar[r]&\biB(G)^H\ar[r]^{\delta^*}&\biB(G)^K\ar[d]\\
&&&\dH^1(G,\biB(G))\ar[d]\\
&&&\vdots}
\end{equation}
As $H$ is not compact, $\biB(G)^H=0$ (cf. Fact~\ref{fact:bg}). Thus \eqref{eq:les} gives an injective map from $\biB(G)^K$ to $\dH^1(G,\biB(G))$. For $K$ is compact, $\biB(G)^K\neq0$ and then $\dH^1(G,\biB(G))\neq 0$.

\noindent \textbf{Case 3.} Let $G=H*_KJ$ and  $H$  non-compact. The sequence
\begin{equation}
\xymatrix{0\ar[r]&\Q[G/K]\ar[r]^-{\delta}&\Q[G/H]\oplus\Q[G/J]\ar[r]&\Q\ar[r]&0,}
\end{equation}
is exact in $\QGdis$ (cf.~Example \ref{ex:ses2}(a)).

Now for $H$ is not compact, applying the long exact cohomology functor with coefficients in $\Q[G/K]$ yields the long exact sequence
\begin{equation}\label{eq:les1}
\xymatrix{0\ar[r]&\Q[G/K]^G\ar[r]&\Hom_G(\Q[G/J],\Q[G/K])\ar[r]^-{\delta^*}&\End_G(\Q[G/K])\ar[d]\\
&&&\dH^1(G,\Q[G/K])\ar[d]\\
&&&\vdots}
\end{equation}
It follows that $\dH^1(G,\Q[G/K])\neq0$. Indeed suppose firstly $J$ to be non-compact. Thus Fact~\ref{fact:bg} implies that  $$\End_G(\Q[G/K])\to\dH^1(G,\Q[G/K])$$ in \eqref{eq:les1} is injective.

On the other hand, if $J$ is compact, we claim that $\delta^*$ cannot be surjective, and so $\dH^1(G,\Q[G/K])\neq0$ as well. 

Let us prove the claim. Recall that the map $\delta$ in \eqref{eq:les1} is given by $$\quad\qquad\delta:\Q[G/K]\to\Q[G/H]\oplus\Q[G/J],\quad\delta(gK)=gH-gJ,\quad \forall g\in G.$$ Let  $\varphi\in\Hom_G(\Q[G/J],\Q[G/K])$, thus one has \begin{equation}
\delta^*(\varphi)(g_1K)=\delta^*(\varphi)(g_2K),
\end{equation}
for all $g_1,g_2\in G$ such that $g_1g_2^{-1}\in J$. If $\delta^*$ is surjective, then there exists $\varphi$ such that $id_{\Q[G/K]}=\delta^*(\varphi)$. But $g_1K\neq g_2K$ for all $g_1,g_2\in G$ such that $g_1g_2^{-1}\in J\setminus K\neq\emptyset$, and the claim follows.

Finally, by Proposition~\ref{prop:fp1} and Remark~\ref{rem:fp},  one has
\begin{equation}
\label{eq:FP-2}
\textstyle{\dH^1(G,\biB(G))=\varinjlim_{\ca{CO}(G)}\dH^1(G,\Q[G/U]),}
\end{equation}
where $U$ is ranging over all compact open subgroups of $G$. Let $\ca{CO}_K(G)$ be the set of all compact open subgroups of $G$ contained in $K$. One has
\begin{equation}
\label{eq:FP-3}
\textstyle{\dH^1(G,\biB(G))=\varinjlim_{\ca{CO}_K(G)}\dH^1(G,\Q[G/U])\neq0,}
\end{equation}
since the map $$\dH^1(\eta_{U,V})\colon \dH^1(G,\Q[G/U])\to\dH^1(G,\Q[G/V])$$ is injective for all compact open subgroups $V\subseteq U\subseteq K$ of $G$ (cf. Proof of \cite[Proposition 4.7]{it:ratdiscoh}) and $\dH^1(G,\Q[G/K])\neq0$.
\end{proof}

In order to conclude the proof of Theorem~A* let us  provide two Lemmas that concur to clarify the expected connection between number of ends and degree--1 cohomology. 

Let $K$ be a compact open subgroup of $G$. Following \cite{dun:acc}, a subset \mbox{$B\subset G/K$} is called an \emph{almost $(G,K)$-invariant set} if the characteristic function $\chi_B$ of $B$ is an almost $(G,K)$-invariant function (cf.~$\S$\ref{ss:aif h1}). 
In other words, $B$ is an almost $(G,K)$-invariant set if $gB=_aB$ (i.e. the symmetric difference is finite) for all $g\in G$ and $kB=B$ for all $k\in K$. Thus we reformulate a result of C. Bamford and M.J. Dunwoody (cf.~\cite[Lemma~1.1]{dunbam}) as follows. 
\begin{lem}\label{lem:gen_ai} 
Let $G$ be a compactly generated t.d.l.c. group and let $(K,S)$ be a generating pair. Then the $\Q$-vector space $\ca{A}Inv_K(G,\Q)$ of all almost $(G,K)$-invariant functions is generated by 
$$\{\chi_B|B\ \mbox{almost $(G,K)$-invariant set}\}.$$
\end{lem}
Note that if $B$ is an almost $(G,K)$-invariant set, then its complement $B^*$ is also an almost $(G,K)$-invariant set. An almost $(G,K)$-invariant set $B\subset G/K$ is said to be \emph{proper} if $B,B^*$ are both infinite.
\begin{lem}\label{lem:proper}
Let $G$ be a compactly generated t.d.l.c. group and let $(K,S)$ be a generating pair of $G$. If there exists a proper almost $(G,K)$-invariant set, then $e(G)>1$.
\end{lem}
\begin{proof}
Let $\Gamma=\Gamma(G,K,S)$ be the rough Cayley graph of $G$ with respect to the generating pair $(K,S)$. If $B\subset G/K$ is an infinite almost $(G,K)$-invariant set, in particular one has $kB=B$ for all $k\in K$. Thus one defines
$$C_B=\{\,gK\in G/K\mid g^{-1}K\in B\,\}\subset \V(\Gamma).$$
Clearly, $C_B$ is  infinite. Moreover, the set $C_B$ has finite boundary. Indeed,
$$\bar{\delta}C_B=\{\,gsK\notin C_B\mid gK\in C_B,\ s\in S\,\}=\{\,s^{-1}g^{-1}K\notin B\mid g^{-1}K\in B,\ s\in S\,\}.$$
Rearranging, we have 
$$\bar{\delta} C_B=\{\,gK\in G/K\mid \exists s\in S\ s.t.\ \chi_B(gK)\neq s\cdot\chi_B(gK)\,\},$$ which is a finite set by the almost invariance of $\chi_B$ and $|S|<\infty$. Clearly, if $B$ is proper then $C_B$ contains at least  a cut of $\Gamma$. Thus Fact \ref{fact:cut_end} completes the proof.
\end{proof}

\begin{proof}[Proof of $c)\Rightarrow a)$] Since
\begin{equation}
\textstyle{\dH^1(G,\biB(G))=\varinjlim_{\ca{CO}(G)}\dH^1(G,\Q[G/U])\not=0,}
\end{equation}
(cf. Proposition~\ref{prop:fp1} and Remark~\ref{rem:fp}), it suffices to prove that $e(G)>1$ if  there exists a compact open subgroup $K$ of $G$ such that \mbox{$\dH^1(G,\Q[G/K])\neq0.$}

Let $K$ be such a subgroup. By Proposition~\ref{prop:almcoh}, there is a non-trivial map $d\in \ca{A}Inv_K(G,\Q)$ which is neither constant on $G/K$ nor almost zero. Since $\ca{A}Inv_K(G,\Q)$ is $\Q$-generated by the characteristic functions of the almost $(G,K)$-invariant sets of $G$ (cf.~Lemma \ref{lem:gen_ai}), there exists an infinite almost $(G,K)$-invariant set $B\subsetneq G/K$. We claim that $B$ is proper. Then the statement follows by Lemma~\ref{lem:proper}.

 Let us prove the claim. Set $B^*=G\setminus B$ and $d^*=g\cdot \chi_{B^*}-\chi_{B^*}$. Clearly,  $d^*\in \dH^1(G,\Q[G/K])$ and $B^*$ is an infinite almost $(G,K)$-invariant set, i.e. $B$ is proper. 
\end{proof}

\section{Compactly presented t.d.l.c. groups of rational discrete cohomological dimension one}\label{s:kps}
Following \cite{it:ratdiscoh}, a {\it graph of profinite groups} $(\ca{G},\Lambda)$ based on the graph $\Lambda$
consists of the following data:
\begin{itemize}
\item[(G1)] a profinite group $\ca{G}_v$ for every vertex $v\in V(\Lambda)$;
\item[(G2)] a profinite group $\ca{G}_{\eue}$ for every edge $\eue\in E(\Lambda)$
satisfying $\ca{G}_{\eue}=\ca{G}_{\bar\eue}$;
\item[(G3)] an open embedding $\iota_{\eue}\colon\ca{G}_{\eue}\rightarrow\ca{G}_{t(\eue)}$ 
for every edge $\eue\in E(\Lambda)$.
\end{itemize}
The fundamental group of a graph of profinite groups carries naturally the structure of t.d.l.c. group. Indeed a neighbourhood basis of the identity is given by
$$\mathcal{B}:=\set{\caO\leq_{co} g\mathcal{G}_vg^{-1} \mid v\in\mathcal{V}(\Lambda),\ g\in\pi_1(\mathcal{G},\Lambda)},$$
where $\caO$ is a compact open subgroup of the vertex stabilizer $g\mathcal{G}_vg^{-1}$. We recall that a {\it generalized presentation} of a t.d.l.c. group $G$ is a graph of profinite groups $(\ca{G},\Lambda)$ together with a continuous open surjective
homomorphism
\begin{equation}
\label{eq:pres}
\phi\colon\pi_1(\ca{G},\Lambda)\longrightarrow G,
\end{equation}
such that $\phi\vert_{\ca{G}_v}$ is injective for all $v\in V(\Lambda)$. In particular, every t.d.l.c. group $G$ admits at least one generalized presentation $(\ca{G},\Lambda_0)$ based on a graph with a single vertex (cf. \cite[Proposition 5.10]{it:ratdiscoh}).
A t.d.l.c. group G is said to be \emph{compactly presented}, if there exists a generalized presentation $((\ca{G},\Lambda), \phi)$, such that
\begin{itemize}
\item[(i)] $\Lambda$ is a finite connected graph, and
\item[(ii)] $ K = ker(\phi)$ is a finitely generated as normal subgroup of the fundamental group $\Pi= \pi_1(\ca{G},\Lambda)$.
\end{itemize}
Clearly, the fundamental group of a finite graph of profinite groups is a compactly presented t.d.l.c. group. 
\begin{rem}
The notion of being compactly presented we use here is equivalent to the usual one defined for compactly generated locally compact groups (cf. \cite[Prop. 1.1.3]{ab:fp}).
\end{rem}
Recall that a compactly generated t.d.l.c. group $G$ is accessible if, and only if, it has an action on a tree $\caT$ such that:
\begin{itemize}
\item[(A1)] the number of orbits of G on the edges of $\caT$ is finite;
\item[(A2)] the stabilizers of edges in $\caT$ are compact open subgroups of G;
\item[(A3)] every stabilizer of a vertex in $\caT$ is a compactly generated open subgroup of G and has
at most one rough end.
\end{itemize}
\begin{thmB}\label{thm:cd1}
Let $G$ be a t.d.l.c. group. Thus the following are equivalent:
\begin{itemize}
\item[(i)] $G$ is a compactly presented t.d.l.c. group with $\ccd_\Q(G)\leq1$,
\item[(ii)] $G$ is isomorphic to the fundamental group $\pi_1(\ca{G},\Lambda)$ of a finite graph of profinite groups $(\ca{G},\Lambda)$.
\end{itemize}
\end{thmB}
\begin{proof} Clearly, the fundamental group $\Pi$ of a finite graph of profinite groups is a compactly presented t.d.l.c. group. Moreover $\Pi$ acts on its universal covering
tree without inversion of edges and with compact open vertex stabilizers. Then $\ccd_\Q(\Pi) \leq 1$ (cf.~\eqref{eq:exseq} and Proposition~\ref{prop:permod}). 

Conversely, let $G$ be a compactly presented t.d.l.c. group. By (a) of Proposition~\ref{prop:cd}, if $\ccd_\Q(G)=0$, then $G$ is profinite and there is nothing to prove. Let $\ccd_\Q(G)=1$. As $G$ is compactly presented, by \cite[Theorem 4.H.1]{corn:qi} $G$ is accessible. Thus $G$ is acting on a tree $\caT$ with finitely many orbits on the set of edges and compact open edge stabilizers. Moreover every vertex stabilizer $G_v$ is a compactly generated open subgroup of $G$ with at most one end. By Theorem~A*, for all $v\in V(\caT)$ one has $\dH^1(G_v,\biB(G_v))=0$.  By Propositions~\ref{prop:fp1} and \ref{prop:cd}(b), $G_v$ is of type $\FP_1$ with $\ccd_{\Q}(\ca{G}_v)\leq1$, i.e., $\ca{G}_v$ is of type $\FP$ for any vertex $v$. Hence Proposition~\ref{prop:FPbg} implies $\ccd_\Q(G_v)=0$ for all $v\in V(\caT)$, i.e., $G_v$ is compact (cf. Proposition \ref{prop:cd}(a)). Finally, Bass-Serre's theory yields the claim. 
\end{proof}
\begin{rem} Clearly Theorem~B can be regarded as the analogue for t.d.l.c. groups of the Karrass-Pietrowski-Solitar theorem for virtually free groups, and in particular of Dunwoody's result \cite[Thm.~1.1]{dun:acc} on accessibility of discrete groups of cohomological dimension one.
\end{rem}


\end{document}